\documentclass{amsart}
\usepackage{amsmath,amsfonts,amsthm, amssymb,color}
\usepackage{hyperref}

\newcommand\bG{{\mathbf G}}
\newcommand\bL{{\mathbf L}}
\newcommand\bS{{\mathbf S}}
\newcommand\bT{{\mathbf T}}
\newcommand\bB{{\mathbf B}}
\newcommand\bU{{\mathbf U}}
\newcommand\bH{{\mathbf H}}
\newcommand\BC{{\mathbb C}}
\newcommand\BF{{\mathbb F}}
\newcommand\BZ{{\mathbb Z}}
\newcommand\BQ{{\mathbb Q}}
\newcommand\CC{{\mathcal C}}
\newcommand\CE{{\mathcal E}}

\newcommand\inv{^{-1}}
\newcommand\Id{\mathrm{Id}}

\DeclareMathOperator\Ind{\mathrm{Ind}}
\DeclareMathOperator\Irr{\mathrm{Irr}}

\DeclareMathOperator\Res{\mathrm{Res}}
\DeclareMathOperator\Sh{\mathrm{Sh}}
\DeclareMathOperator\diag{\mathrm{diag}}
\DeclareMathOperator\SL{\mathrm{SL}}
\DeclareMathOperator\SU{\mathrm{SU}}
\DeclareMathOperator\GL{\mathrm{GL}}
\DeclareMathOperator\PGL{\mathrm{PGL}}
\DeclareMathOperator\ad{\mathrm{ad}}
\DeclareMathOperator\rank{{rank}}
\newtheorem{assumption}[equation]{Assumption}
\newtheorem{lemma}[equation]{Lemma}
\newtheorem{definition}[equation]{Definition}
\newtheorem{theorem}[equation]{Theorem}
\newtheorem{proposition}[equation]{Proposition}

\newtheorem{conjecture}[equation]{Conjecture}

\theoremstyle{remark}

\newcommand\scal[2]{\langle#1,#2\rangle}
\newcommand\genby[1]{\langle#1\rangle}
\newcommand\fS{{\mathfrak S}}
\newcommand\Fqbar{{\overline\BF_q}}
\newcommand\Fq{{\BF_q}}
\newcommand\GF{{\bG^F}}
\newcommand\UF{{\bU^F}}
\newcommand\HF{{\bH^F}}
\newcommand\TF{{\bT^F}}
\newcommand\Gd{{\bG^*}}
\newcommand\Gtd{{{\tilde\bG}^*}}
\newcommand\Td{{\bT^*}}
\newcommand\GdFd{{\bG^{*F^*}}}
\newcommand\RTGd{{R^\bG_\Td}}
\newcommand\Fd{{F^*}}
\newcommand\conn{^0}
\newcommand\CGdFs{{C_\Gd(s)^\Fd}}

\newcommand\lexp[2]{\kern\scriptspace\vphantom{#2}^{#1}\kern-\scriptspace#2}
\newcommand\xdownarrow[1]{\Big\downarrow\rlap{$\vcenter{\hbox{$\scriptstyle#1$}}$}}
\newcommand\ve\varepsilon
\newcommand\Fe{F_\ve}
\author{Fran\c cois Digne and Jean Michel}
\title{Commutation of Shintani descent and Jordan decomposition}
\begin{document}
\dedicatory{Dedicated to the memory of T.A. Springer}
\begin{abstract}
Let ${\mathbf G}^F$ be a finite group of Lie type, where ${\mathbf G}$ is a reductive group
defined  over  ${\overline{\mathbb F}_q}$  and  $F$  is  a  Frobenius  root. Lusztig's Jordan
decomposition  parametrises the irreducible characters in a rational series
${\mathcal E}({{\mathbf G}^F},(s)_{{\mathbf G}^{*F^*}})$ where $s\in{{\mathbf G}^{*F^*}}$ by the series ${\mathcal E}(C_{{\mathbf G}^*}(s)^{F^*},1)$.
We conjecture  that  the  Shintani  twisting preserves the space of
class functions generated by the union of the ${\mathcal E}({{\mathbf G}^F},(s')_{{\mathbf G}^{*F^*}})$ where
$(s')_{{\mathbf G}^{*F^*}}$ runs
over  the semi-simple  classes of  ${{\mathbf G}^{*F^*}}$ geometrically  conjugate to $s$;
further,  extending the Jordan decomposition by linearity to this space, we
conjecture  that  there is a way to fix Jordan  decomposition  such that it
maps  the Shintani twisting  to the Shintani
twisting  on disconnected  groups defined  by Deshpande, which acts  on the
linear span of $\coprod_{s'}{\mathcal E}(C_{{\mathbf G}^*}(s')^{F^*},1)$. We show a non-trivial
case of this conjecture, the case where ${\mathbf G}$ is of type $A_{n-1}$
with $n$ prime.
\end{abstract}
\maketitle
\section{Deshpande's approach to Shintani descent}\label{Deshpande}
We follow \cite{deshpande}.
Let $\bH$ be an algebraic group over an algebraically closed field $k$; we
identify $\bH$ to its points $\bH(k)$ over $k$. Let
$\gamma_1, \gamma_2$ be two commuting bijective
isogenies on $\bH$.  We define the following subset of
$(\bH\rtimes\genby{\gamma_1})\times(\bH\rtimes\genby{\gamma_2})$:
$$R_{\gamma_1,\gamma_2}=\{(x\gamma_1,y\gamma_2)\mid x,y\in\bH,
[x\gamma_1,y\gamma_2]=1\}$$
where $[u,v]$ is the commutator $uvu\inv v\inv$.
A matrix $\begin{pmatrix}a&b\\c&d\end{pmatrix}\in\GL_2(\BZ)$ defines a map
 $$R_{\gamma_1,\gamma_2}\xrightarrow{\begin{pmatrix}a&b\\c&d\end{pmatrix}}
   R_{\gamma_1^a\gamma_2^c,\gamma_1^b\gamma_2^d}:
(x\gamma_1,y\gamma_2)\mapsto((x\gamma_1)^a(y\gamma_2)^c,
(x\gamma_1)^b(y\gamma_2)^d)$$
There is an action of $\bH$ on $R_{\gamma_1,\gamma_2}$ by simultaneous
conjugation: $(x\gamma_1,y\gamma_2)\xrightarrow{\ad
g}(gx\gamma_1g\inv,gy\gamma_2g\inv)$, which commutes with the maps
$\begin{pmatrix}a&b\\c&d\end{pmatrix}$, and we denote by
$R_{\gamma_1,\gamma_2}/\sim\bH$ the space of orbits under this action.

Assume  now that  $\bH$ is  connected and  $\gamma_1$ is  a Frobenius root
(an isogeny such that some finite power is a Frobenius morphism).
Then,  by the Lang-Steinberg theorem any  $x\gamma_1\in\bH\gamma_1$
is $\bH$-conjugate to $\gamma_1$,
and   we  can  take  as  representatives  of  the  $\bH$-orbits  pairs
of the form $(\gamma_1,y\gamma_2)$; on these pairs there is only an action of
the fixator of $\gamma_1$, that is $\bH^{\gamma_1}$. Further, the condition
$[\gamma_1,y\gamma_2]=1$  is equivalent to  $y\in\bH^{\gamma_1}$ (since
$\gamma_1$ and $\gamma_2$ commute).  We can thus
interpret $R_{\gamma_1,\gamma_2}$ as the $\bH^{\gamma_1}$-conjugacy classes
on the coset $\bH^{\gamma_1}\gamma_2$, which we denote by
$\bH^{\gamma_1}\gamma_2/\sim\bH^{\gamma_1}$. If $a$, $c$ and $\gamma_2$
are such that $\gamma_1^a\gamma_2^c$ is still a Frobenius root, we can thus
interpret the map given by $\begin{pmatrix}a&b\\c&d\end{pmatrix}$ as
a map $\bH^{\gamma_1}\gamma_2/\sim\bH^{\gamma_1}\to
\bH^{\gamma_1^a\gamma_2^c}\gamma_1^b\gamma_2^d/\sim\bH^{\gamma_1^a\gamma_2^c}$,
a ``generalised Shintani descent''.

We are interested here in the case of $R_{F,\Id}$ where $F$ is a Frobenius
root. The matrix $\begin{pmatrix}1&0\\ 1&1\end{pmatrix}$ defines a
map, that we call the Shintani twisting and will just denote by 
$\Sh: R_{F,\Id}\to R_{F,\Id}:
(xF,y)\mapsto(xFy,y)$. With the identification above of $R_{F,\Id}/\sim\bH$
with the pairs $(F,y)$ which identifies it with $\HF/\sim\HF$,  if we write
$y=\lambda.\lexp F\lambda\inv$ using the Lang-Steinberg theorem, $\Sh$ maps $(F,y)$ to
$(Fy,y)=(F\lambda.\lexp F\lambda\inv,\lambda.\lexp F\lambda\inv)$
which is conjugate by $\lexp F\lambda\inv$ to $(F,\lexp F\lambda\inv\lambda)$,
thus we recover the usual definition of the Shintani twisting $\Sh$ as being
induced by the
correspondence $\lambda\lexp F\lambda\inv\mapsto\lexp F\lambda\inv\lambda$.

The  advantage of Deshpande's  approach is that  the map $\Sh: R_{F,\Id}\to
R_{F,\Id}$  still makes  sense when  $\bH$ is  disconnected; this  time the
interpretation   of  $R_{F,\Id}/\sim\bH$  is  different:  the
$\bH$-orbits  on  $\bH  F$  are  parametrised by $H^1(F,\bH/\bH\conn)$. If
$\sigma\in\bH$  is a representative of such an $\bH$-orbit,
the $\bH$-orbits  of pairs  $(\sigma F,y)$
are in  bijection with  $\bH^{\sigma F}/\sim
\bH^{\sigma  F}$, so  we see  that we  must consider  all rational forms of
$\bH$ corresponding to the various representatives $\sigma
F\in H^1(F,\bH/\bH\conn)$ together, and $\Sh$
will act on the disjoint union of the conjugacy classes of each of these
forms. Let $\sigma F$ be such a representative, let $\bH_1$ be a $\sigma
F$-stable coset of $\bH\conn$ in $\bH$, let $\sigma'\in\bH_1^{\sigma F}$
and let us compute the image by $\Sh$ of a commuting pair $(\sigma F,\sigma' y)$
where $y\in\bH\conn$ (thus $y\in(\bH\conn)^{\sigma F}$).
Let us write using the Lang-Steinberg
theorem $y=\lambda.\lexp{\sigma'\sigma F}\lambda\inv$. Then $\Sh$ maps
$(\sigma F,\sigma' y)$ to $(\sigma F\sigma' y,\sigma' y)=
(\sigma'\sigma F\lambda.\lexp{\sigma'\sigma F}\lambda\inv,
\sigma'\lambda.\lexp{\sigma'\sigma F}\lambda\inv)$ (using
$[\sigma F,\sigma']=1$) which is conjugate by $\lexp{\sigma'\sigma F}\lambda\inv$
to $(\sigma'\sigma F,\sigma'.\lexp{\sigma F}\lambda\inv\lambda)$. Thus 
\begin{proposition}\label{Sh}
$\Sh$ is induced on $\sigma'(\bH\conn)^{\sigma F}$, where
$\sigma$ and $\sigma'$ are two representatives of $\bH/\bH\conn$ such that
 $\sigma F$ and $\sigma'$ commute,
by the correspondence between the element
$\sigma'\lambda.\lexp{\sigma'\sigma F}\lambda\inv$ of
$\sigma'(\bH\conn)^{\sigma F}$
and the element $\sigma'.\lexp{\sigma F}\lambda\inv\lambda$ of
$\sigma'(\bH\conn)^{\sigma'\sigma F}$.
\end{proposition}

Assume now that $\bH/\bH\conn$ is commutative
and that we can lift all elements of $\bH/\bH\conn$ to commuting
representatives, such that $F$-stable elements lift to $F$-stable
representatives; then we can lift all pairs $(\sigma F,\sigma')$ as in 
Proposition \ref{Sh} which commute in $\bH/\bH\conn$ to commuting pairs,
thus we can see $\Sh$ as a linear map
on the space $\CC(\bH,F):=\oplus_{\sigma\in H^1(F,\bH/\bH\conn)}
\CC(\bH^{\sigma F})$, the direct sum of the spaces of class functions on
the various rational forms $\bH^{\sigma F}$. 

\begin{definition} Let $\bH$ be a algebraic group with a Frobenius
root $F$, then we say
that $H$ satisfies condition (*) if we can lift elements of $\bH/\bH\conn$
to commuting elements of $\bH$, such that $F$-stable elements of $\bH/\bH\conn$
lift to $F$-stable elements.
\end{definition}

Beware that $\Sh$ on class functions is defined by
$\Sh(f)(x)=f(\Sh(x))$ thus it maps class functions on
$\sigma'(\bH\conn)^{\sigma'\sigma F}$ to class functions on
$\sigma'(\bH\conn)^{\sigma F}$.

Note  finally that  when $\bH\conn$  is trivial  and $F$  acts trivially on
$\bH/\bH\conn$,  the computations  of this  section recover  the well-known
action  of  $\GL_2(\BZ)$  on  the  Drinfeld  double  of  the  finite  group
$\bH/\bH\conn=\bH$. See, for example \cite [8.4.2]{Broue}.
\section{Conjectures}
\begin{definition}\label{RTG non connected}
We extend the definition of Deligne-Lusztig characters to a disconnected
reductive group $\bH$ with a Frobenius root $F$ by $R_\bT^\bH(s):=
\Ind_{\bH^{0F}}^\HF R_\bT^{\bH^0}(s)$
where $\bT$ runs over $F$-stable maximal
 tori of $\bH$ and $s$ over $\Td^\Fd$. 
We call unipotent characters the irreducible components of the $R_\bT^\bH(1)$ 
and we denote by $\CE(\HF,1)$ the set of unipotent characters  of $\HF$.
\end{definition}

From now on, $\bG$ will be a connected reductive group with a Frobenius root
$F$ and $s$ an $\Fd$-stable semi-simple element
of $\Gd$, the group dual to $\bG$. We denote by 
$\CE(\GF,(s)_\GdFd)$ the $\GdFd$-Lusztig series
associated with $s$; this series consists of
the irreducible components of the $\RTGd(s)$ where $\Td$ runs over $F$-stable
maximal tori of $C_\Gd(s)$. We denote by 
$\CE(\GF,(s))$ the geometric Lusztig series associated with $s$, that is the
union of the series $\CE(\GF,(s')_\GdFd)$ where $(s')_\GdFd$ runs over the 
$\GdFd$-classes in the geometric class of $s$.

Let  $\varepsilon_\bG:=(-1)^{F-\rank\bG}$.  Lusztig's  Jordan  decomposition of
characters  in a reductive group with a not necessarily connected centre is
given   by   the   following   theorem   (see   for  example  \cite[Theorem
11.5.1]{book2}).
\begin{theorem}\label{Jordan decomposition} Let $\bG$ be a connected reductive group with
Frobenius root $F$ and $(\Gd,\Fd)$ be dual to $\bG$;
for any semi-simple element $s\in\GdFd$, there is a bijection from
$\CE(\GF,(s)_\GdFd)$ to $\CE(\CGdFs,1)$. This bijection may
be chosen such that, extended by linearity to virtual characters, it sends
$\varepsilon_\bG \RTGd(s)$  to
$\varepsilon_{C_\Gd(s)^0}R_\Td^{C_\Gd(s)}(1)$ for any $\Fd$-stable
maximal torus $\Td$ of $C_\Gd(s)$.
\end{theorem}
\begin{conjecture}\label{Sh preserve series}
Let $\bG$ be a connected reductive group with Frobenius root $F$. Then
$\Sh$, viewed as a linear operator on $\CC(\GF)$, preserves the subspace
spanned by $\CE(\GF,(s))$ for each geometric class $(s)$.
\end{conjecture}
\begin{proposition}Assume that $F$ is the Frobenius
morphism corresponding to an $\Fq$-structure on $\bG$. Then
Conjecture  \ref{Sh  preserve  series}  holds when the
characteristic is almost good for $\bG$ in the sense of
\cite[1.12]{lusztig green functions} and $Z\bG$ is connected or when $\bG$ is of
type $A$ and either $q$ is large or $F$ acts trivially on $Z\bG/(Z\bG)\conn$.
\end{proposition}
\begin{proof}
The case of type $A$ when $F$ acts trivially on $Z\bG/(Z\bG)\conn$
results from \cite[Th\'eor\`eme 5.5.4]{bonnafe}. Let us prove the other cases.
In \cite[Section 3.2]{Sh} it is shown
that  the characteristic functions of 
character sheaves are eigenvectors  of $\Sh$; then in \cite[Theorem 5.7]{Sh}
and \cite[Theorem 3.2 and Theorem 4.1]{Sh2}
(resp.\ \cite[Corollaire 24.11]{bonnafe})
it is shown that when $Z\bG$ is connected and the characteristic is almost
good (resp.\ in type $A$ with $q$ large) the
characteristic  functions  of  character  sheaves  coincide up to scalars with
``almost  characters''.  Since  for  each  $s$, the space $\BC\CE(\GF,(s))$ is
spanned  by a subset of almost  characters, this gives the result. \end{proof}
It is probable that the above proof applies to more reductive groups (see
\cite{Sh3}, \cite{W}, and also \cite[Introduction]{bonnafe}), but it is difficult to find
appropriate statements in the literature.

Assuming conjecture \ref{Sh preserve series}, 
a choice of a Jordan decomposition as specified by
Theorem \ref{Jordan decomposition} will map the operator $\Sh$ on $\BC\CE(\GF,(s))$
to a linear operator on the
space $$\CC(\CGdFs,\Fd,1):=\oplus_{\sigma\in H^1(\Fd,C_\Gd(s)/C_\Gd\conn(s))}
\BC\CE(C_\Gd(s)^{\sigma F},1).$$
\begin{conjecture}\label{Sh commute Jordan}
Let $s\in\GdFd$. Then $C_\Gd(s)$ satisfies (*) and
$\Sh$ on $\CC(C_\Gd(s),\Fd)$ preserves the subspace $\CC(\CGdFs,\Fd,1)$, and
the choice of a Jordan decomposition in Theorem \ref{Jordan decomposition}
may be refined so that it maps $\Sh$ on 
$\CC\CE(\GF,(s))$ to $\Sh$ on $\CC(\CGdFs,\Fd,1)$.
\end{conjecture}

\section{The case of a group of type $A_{n-1}$ when $n$ is prime.}

\begin{proposition}
Conjectures \ref{Sh preserve series} and 
 \ref{Sh commute Jordan} hold for any reductive group of type $A_{n-1}$, with
$n$ prime, if they hold for $\SL_n$. 
\end{proposition}
\begin{proof}
We start with two lemmas.
\begin{lemma}\label{restriction} Let $\tilde\bG$ be a connected reductive group
with a Frobenius root $F$. 
Let $\bG$ be a closed connected $F$-stable subgroup of $\tilde\bG$
with same derived group.
Let $\Fd$ be a Frobenius root dual to $F$ and let $s\in\Gd^\Fd$ 
be such that $C_\Gd(s)^\Fd=C_\Gd\conn(s)^\Fd$. Then the characters in
$\CE(\GF,(s)_{\Gd^\Fd})$ are the restrictions from $\tilde\bG^F$ to $\GF$ of
the characters in $\CE(\tilde\bG^F,(\tilde s)_{\tilde\bG^{*\Fd}})$ where 
$\tilde s\in{\tilde\bG}^{*\Fd}$ lifts $s\in\Gd^\Fd$.
\end{lemma}
\begin{proof}
By definition the characters in $\CE(\tilde\bG^F,(\tilde s)_{\tilde\bG^{*\Fd}})$
are those which occur in some $R_{\tilde\bT}^{\tilde\bG}(\tilde s)$, where 
$\tilde\bT$ is a maximal torus such that $\tilde\bT^*\owns\tilde s$,
and by for instance \cite[11.3.10]{book2} we have
 $\Res^{\tilde\bG^F}_\GF R_{\tilde\bT}^{\tilde\bG}(\tilde s)=
 R_{\tilde\bT\cap\bG}^\bG(s)$.
It is thus sufficient for proving the lemma to prove that for any 
$\tilde\chi\in\CE(\tilde\bG^F,(\tilde s)_{\tilde\bG^{*\Fd}})$ its restriction
$\chi$ to $\GF$ is irreducible.

By, for instance \cite[11.3.9]{book2} we have
 $\langle\chi,\chi\rangle_\GF=|\{\theta\in\Irr(\tilde\bG^F/\GF)\mid
\tilde\chi\theta=\tilde\chi\}|$. It is thus sufficient to prove that for
any non-trivial such $\theta$ we have $\tilde\chi\theta\ne\tilde\chi$.
But a non-trivial $\theta$ corresponds to a non-trivial $z\in
Z(\Gtd)^\Fd$ with trivial image in $\Gd$, and since $\tilde\chi\theta
\in\CE(\tilde\bG^F,(\tilde sz)_{\tilde\bG^{*\Fd}})$
it is sufficient to prove that $\tilde s$ and $\tilde sz$ are not 
$\tilde\bG^{*\Fd}$-conjugate. But, indeed
if $\lexp g(\tilde sz)=\tilde s$ then $\lexp{\bar g}s=s$ where $\bar g$ is
the image of $g$ in $\Gd$, thus $\bar g\in C_\Gd(s)^\Fd=C_\Gd\conn(s)^\Fd$, 
whence $g\in C_\Gtd(\tilde s)$ (a contradiction) since the preimage of 
$C_\Gd\conn(s)$ is $C_\Gtd\conn(\tilde s)$ (by for example \cite[3.5.1
 (i)]{book2}).
\end{proof}

\begin{lemma}\label{quotient}
Consider a quotient $1\to Z\to\bG_1\xrightarrow\pi\bG\to 1$, where
$Z$ is a connected subgroup of $\bG_1$.
Assume that $\bG$ and $\bG_1$
have Frobenius roots both denoted by $F$ and 
that $\pi$ commutes with $F$; 
then $\pi$ is surjective
from the conjugacy classes of $\bG^F_1$ to that of $\GF$, and
commutes with $\Sh$.
\end{lemma}
\begin{proof}
The group $Z$ is $F$-stable since $\pi$ commutes with $F$.
Let $x\in\GF$. Since $Z$ is connected, by the Lang-Steinberg theorem
$\pi\inv(x)^F$ is non-empty whence
the surjectivity of $\pi$ on elements, hence on conjugacy classes.
If $y\in\pi\inv(x)^F$ the image by $\pi$ of $\Sh(y)$ is $\Sh(x)$, whence the
commutation of $\Sh$ with $\pi$.
\end{proof}
Let now $\bG$ be an arbitrary reductive group of type $A_{n-1}$.
Since $n$ is prime, the only semi-simple connected reductive groups
of type $A_{n-1}$ are $\SL_n$ and $\PGL_n$; thus any connected reductive
group $\bG$ of type $A_{n-1}$ is the almost direct product of the derived
$\bG'$ of $\bG$, equal to $\SL_n$ or $\PGL_n$, by a torus $\bS$.

In the case $\bG'=\PGL_n$ the almost direct product is direct since $\PGL_n$ has a trivial
centre. If $\bT$ is an $F$-stable maximal
torus of $\bG$, then it has a decomposition $\bT_1\times\bS$ where
$\bT_1=\bT\cap\bG'$ is $F$-stable since $\bG'$ is $F$-stable. It is possible to find an $F$-stable
complement $\bS'$ of $\bT_1$ (see for example the proof of
\cite[2.2]{EM})  and then $\bG$ has an $F$-stable product
decomposition $\bG'\times\bS'$. Since $\Sh$ is trivial on a torus the
conjectures are reduced to the case of $\bG'=\PGL_n$.

The same argument reduces the case of a direct product $\bG=\bG'\times\bS$
where $\bG'=\SL_n$ to the case of $\SL_n$. The other possibility when
$\bG'=\SL_n$ is an almost direct product
$\SL_n\times\bS$ amalgamated by $Z\SL_n$; this is isomorphic to a product
of the form $\GL_n\times\bS'$; in such a group all centralisers are connected,
as well as in the dual group (isomorphic to $\bG$) thus both conjectures are
trivial since $\Sh$ is the identity (by for example \cite[IV, 1.1]{DMthese}).

Finally, in $\PGL_n$ the action of $\Sh$ is trivial by  Lemma \ref{quotient}
applied to the quotient $\GL_n\xrightarrow\pi\PGL_n$, 
and in the dual $\SL_n$ semisimple elements have connected centralisers which
are Levi subgroups of $\SL_n$; considering the embedding of these centralisers
in the corresponding Levi of $\GL_n$, on which $\Sh$ is trivial since it is
isomorphic to a product of $\GL_{n_i}$ we get  that $\Sh$ is trivial on the 
unipotent characters of these centralisers by Lemma \ref{restriction}.
\end{proof}
We have thus shown that we only need to consider $\SL_n$.
Up to isomorphism there are
two possible $\Fq$-structures on $\SL_n$ (only one if $n=2$) thus $F$
will be one of the Frobenius endomorphisms
$F_+$ or $F_-$ where $\SL_n^{F_+}=\SL_n(\Fq)$ and $\SL_n^{F_-}=\SU_n(\Fq)$.
When we want to consider both cases simultaneously but keep track whether
$F=F_+$ or $F=F_-$, we will denote the Frobenius by
$F_\ve$ with $\ve\in\{-1,1\}$, where we always take $\ve=1$ if $n=2$.
We will use the dual group $\PGL_n$,
the inclusion $\SL_n\subset\GL_n$ and the quotient $\GL_n\to\PGL_n$.
Since $(\GL_n,F)$ is its own dual, we will write $F$ (instead of
$F^*$) for the Frobenius map on the dual of $\GL_n$ and on $\PGL_n$.
We choose for $F_+$ the standard Frobenius which raises all matrix entries to
the $q$-th power, and choose for $F_-$ the map given by
$x\mapsto F_+(\lexp t(x\inv))$. This choice is such that
on the torus of diagonal matrices $\Td$ of $\Gd$, $F_-$ acts by raising all the
eigenvalues to the power $-q$, and acts trivially on $W_\Gd(\Td)$.The torus
$\Td$ is split for $F_+$ and of type $w_0$ (the longest element of the Weyl
group) with respect to a quasi-split torus for $F_-$.

\begin{proposition}\label{series s only}
For the group $\SL_n$, $n$ prime, with Frobenius $F=F_\ve$ conjecture
\ref{Sh preserve series} holds. Further, conjecture \ref{Sh commute Jordan} 
holds if it holds when $q\equiv \ve \pmod n$ for the series $\CE(\SL^F_n,
(s))$ with $s\in\PGL^F_n$ geometrically conjugate to
$\diag(1,\zeta,\zeta^2,\ldots,\zeta^{n-1})$ where $\zeta\in\Fqbar$
is a non-trivial $n$-th root of 1.
\end{proposition}
\begin{proof}
Lemma \ref{restriction} applied with $\bG=\SL_n$ and $\tilde\bG=\GL_n$
shows that conjectures \ref{Sh preserve series} and 
\ref{Sh commute Jordan} hold when $C_{\PGL_n}(s)^F=C_{\PGL_n}\conn(s)^F$.
Indeed, in this case,
since $\Sh$ is trivial in $\GL_n$ and commutes with
the restriction to $\SL_n$ by Lemma \ref{quotient}, 
Lemma \ref{restriction} shows that $\Sh$ is trivial
on $\CE(\SL_n^F,(s))$; on the other hand, since $C_{\PGL_n}\conn(s)$
is the quotient with
central kernel of a product of $\GL_{n_i}$, the unipotent characters of 
$C_{\PGL_n}(s)^F=C_{\PGL_n}\conn(s)^F$ can be lifted to this product
(see \cite[Proposition 11.3.8]{book2}). Since $\Sh$ is trivial in $\GL_{n_i}$
and commutes which such a quotient with connected central kernel by
Lemma \ref{quotient},
it is trivial on the unipotent characters of $C_{\PGL_n}(s)^F$.

We have thus reduced the study of conjectures  \ref{Sh preserve series} and
\ref{Sh commute Jordan} to the case of semi-simple elements $s\in\PGL_n$ with
$C_{\PGL_n}(s)^F\neq C_{\PGL_n}\conn(s)^F$.

By \cite[Lemma 11.2.1(iii)]{book2}, since $n$ is prime
a semi-simple element $s\in\PGL_n$ has
a non-connected centraliser if and only if 
$|C_{\PGL_n}(s)/C_{\PGL_n}\conn(s)|=n$.
Since this group of components is a subquotient of the Weyl group $\fS_n$,
it has order $n$ if and only if it is a subgroup of $\fS_n$ generated by an 
$n$-cycle.
An easy computation shows that such an
 $s$ is as in the statement of Proposition \ref{series s only}.
Now the centraliser of $s$ is the semidirect product of $\Td$ with the
cyclic group generated by the $n$-cycle $c=(1,\ldots,n)$. Assume that
the geometric class of $s$
has an $F$-stable representative $s'$ in a torus of type $w\in\fS_n$
with respect to $\Td$ (or equivalently that $\lexp{wF}s=s$).
Then $C_{\PGL_n}(s')^F\simeq C_{\PGL_n}(s)^{wF}$ hence is equal to
$C_{\PGL_n}\conn(s')^F$ unless
$w$ commutes with $c$, that is $w=c^i$ for some $i$. 
Since $c$ centralises $s$ we get that $s$ is $F$-stable, which means for
$F=F_\ve$ that $n$ divides $q-\varepsilon$, in particular Conjecture 
\ref{Sh commute Jordan} holds when $q\not\equiv \varepsilon\pmod n$;
we recall that we always take $\ve=1$ when $n=2$.

Finally, conjecture \ref{Sh preserve series}
holds in any case, since $\Sh$ preserves all geometric series except
possibly that of $s$ and preserves orthogonality of characters, thus preserves
 also the series $\CE(\SL_n^F,(s))$.
\end{proof}
Henceforth we assume that $n$ is  prime and divides $q-\ve$, and study
conjecture \ref{Sh commute Jordan} for the particular $s$ of Proposition
\ref{series s only}.

\begin{assumption}\label{assumptions}
We choose $\zeta$, a primitive $n$-th root of unity in $\Fqbar$.
We choose an isomorphism $\Fqbar^\times\simeq(\BQ/\BZ)_{p'}$ which maps
$\zeta$ to $1/n$ and a group embedding $\BQ/\BZ\hookrightarrow\BC^\times$
 which maps $1/n$ to $\zeta_\BC:=\exp(2 i\pi/n)$.
\end{assumption}

We denote by $\bT$ the diagonal torus of $\SL_n$ and by $\tilde\bT$ the diagonal
torus of $\GL_n$; we choose the dual torus $\Td$ to be the
diagonal torus of $\PGL_n$. We let $s\in\Td$ be the image
of $\diag(1,\zeta,\zeta^2,\ldots,\zeta^{n-1})\in\tilde\bT$.

\vskip 6mm
\subsection*{\hfill$\Sh$ on $C_{\PGL_n}(s)$\hfill}
\hfill\break

We first study $\Sh$ on $\CC(C_{\PGL_n}(s)^F,F,1)$.
We have $C_{\PGL_n}(s)=\Td\rtimes \genby c$ where $c$
is the permutation matrix representing the cycle
$(1,2,\ldots,n)\in\fS_n$, which acts on $\Td$ by sending
$\diag(t_1,\ldots,t_n)$ to $\diag(t_n,t_1,\ldots,t_{n-1})$. 
On $\Td$  the Frobenius $\Fe$ acts by $t\mapsto t^{\ve q}$, and acts trivially
on $W_\Gd(\Td)$.
Note that $C_\Gd(s)$ satisfies condition (*) since $c$ is $F$-stable.
Since $C_{\PGL_n}(s)/C_{\PGL_n}\conn(s)=\genby c$, and $F$ acts
trivially on $c$,  we have $H^1(F,C_{\PGL_n}(s)/C_{\PGL_n}\conn(s))=\genby c$ 
and the
geometric class of $s$ splits into $n$ rational classes  parametrised by the
powers of $c$.  A representative of the class parametrised by $c^j$ is 
$\lexp xs$ where $x$ is such that $x\inv\lexp Fx=c^j$. This representative
lies in a maximal torus $\bT^*_{c^j}=\lexp x\Td$ of
type $c^j$ with respect to $\Td$. Choosing the ${\PGL_n}$-conjugacy by $x\inv$ 
to identify $(\bT^*_{c^j},F)$ with $(\Td,c^jF)$, we
identify back the representative of the class parametrised by $c^j$
with the element $s$ in $\bT^{*c^jF}$.  We have
$$\CC(C_{\PGL_n}(s)^F,F,1)=\oplus_{j=0}^{n-1}\BC\CE(\Td^{c^jF}\rtimes\genby
c,1).$$
Since $\Td$ is a torus, the
unipotent characters of $\bT^{*c^jF}\rtimes\genby c$
are the $n$ extensions of the trivial character of $\bT^{*c^jF}$.
We parametrise these by $Z\SL_n\times\BZ/n\BZ$ in the
following way: if $z_0=\diag(\zeta,\ldots,\zeta)$, we call $\theta_{z_0^i,k}$ the character of 
$\bT^{*c^kF}\rtimes\genby c$  
which is trivial on $\bT^{*c^kF}$ and equal to $\zeta_\BC^i$ on $c$.
This allows us to define another basis of
$\CC(C_{\PGL_n}(s)^F,F,1)$, the ``Mellin transforms'', defined
for $j,k\in\BZ/n\BZ$ by $\theta_{j,k}:= \sum_{z\in Z\SL_n}\hat
s_j(z)\theta_{z,k}$, where $\hat s_j$ is the character of $\bT^{c^jF}$
corresponding to the element $s\in\bT^{*c^jF}$ through duality.
The point is that it is more convenient to compute the action of $\Sh$ on 
the Mellin transforms:
\begin{proposition}\label{Sh theta}We have $\Sh\theta_{j,i}=\theta_{j,i+j}$
 unless $n=2$ (thus $\ve=1$) and $q\equiv -1\pmod 4$.
In this last case we have $\Sh\theta_{j,i}=\theta_{j,i+j-1}$
\end{proposition}
\begin{proof}
\begin{lemma}\label{value of hat s} With the conventions of Assumption
\ref{assumptions}, when $n$ is odd,
 we have $\hat s_j(z_0^i)=\zeta_\BC^{ij}$; when $n=2$  (thus $\ve=1$) we have
$\hat s_j(z_0)=\begin{cases}
(-1)^{(q-1)/2}&\text{ if }j=0\\
(-1)^{(q+1)/2}&\text{ if }j=1
\end{cases}$.
\end{lemma}
\begin{proof}
The correspondence between $\bT^{*c^jF}$ and
 $\Irr(\bT^{c^jF})$ comes from the diagram of exact sequences
$$\begin{matrix}
X(\bT)&\xrightarrow{c^jF-1}&X(\bT)&\xrightarrow\Res
 &\Irr(\bT^{c^jF})&\rightarrow 1\cr
\big\Vert&&\big\Vert&&\xdownarrow\sim&\cr
 Y(\Td)&\xrightarrow{c^jF-1}&Y(\Td)&\xrightarrow\pi&\bT^{*c^jF}&\rightarrow 1$$
\end{matrix}
$$
where $\pi$ maps $y\in Y(\Td)$ to $N_{F^m/c^jF}(y(1/(q^m-1)))$ for any $m$
such that $(c^jF)^m=F^m$ (see \cite[Proposition 11.1.7]{book2}); here
$y(1/(q^m-1))$ is defined using the choices of Assumption \ref{assumptions}
and $N_{F^m/c^jF}$ is the
norm map on $\Td$ given by $x\mapsto x.\lexp {c^jF} x\ldots
\lexp{(c^jF)^{m-1}}x$. 

When $j=0$ we can take $m=2$ in the above diagram. We have
$s=N_{F^2/F}(y(1/(q^2-1)))=y(1/(\ve q-1))$ where
$y$ is the cocharacter mapping $\lambda\in\Fqbar^\times$ to the image in $\bT^*$ of
$$\diag(1,\lambda^{(\ve q-1)/n},\lambda^{2(\ve
 q-1)/n},\ldots,\lambda^{(n-1)(\ve q-1)/n}) \in\tilde\bT.$$
The element $y$ interpreted as a character of $\TF$ maps
$z_0=\diag(\zeta,\ldots,\zeta)\in Z\SL_n$ to
$\zeta_\BC^{\frac{\ve q-1}n\frac{n(n-1)}2}=\zeta_\BC^{(n-1)(\ve q-1)/2}$ which
is equal to $1$ when $n$
is odd---thus in this case  $\hat s_0$ is trivial on $Z\SL_n$. If $n=2$ 
we have $\ve=1$ and we get $\hat s_0(\diag(-1,-1))=(-1)^{(q-1)/2}$.

When $j\neq 0$, we can take $m=2n$. We want to find $y\in Y(\Td)$ such that
 $s=N_{F^{2n}/c^jF}(y(1/(q^{2n}-1)))$. Since $(c^j F)^n=F^n$ we have 
$$N_{F^{2n}/c^jF}(x)=N_{F^n/c^jF}(x)F^n(N_{F^n/c^jF}(x)).$$
Let $y_0$ be a cocharacter mapping
$\lambda\in\BF_{q^{2n}}^\times$ to the image in $\bT^*$ of
$\diag(\lambda,1,\ldots,1)$.  We have
$(c^jF)^k(\lambda,1,\dots,1)=(1,\ldots,1,\lambda^{(\ve q)^k},1,\ldots)$ where 
$\lambda^{(\ve q)^k}$ is at the place indexed by $kj+1$, hence at the place indexed
by $i$ in $N_{F^n/c^jF}(y_0(\lambda))$ the exponent of $\lambda$ is such that 
$i\equiv kj+1\pmod n$, which since $\lambda\in\BF_{q^n}^\times$ is equivalent
to $k =(i-1)j'$ where $jj'\equiv1\pmod n$. Hence
$N_{F^n/c^jF}(y_0(\lambda))$ is the image in $\bT^*$ of
 $\diag(\lambda,\lambda^{(\ve q)^{j'}},\lambda^{(\ve
 q)^{2j'}},\ldots,\lambda^{(\ve q)^{(n-1)j'}})$ and 
 $N_{F^{2n}/c^jF}(y_0(\lambda))$ is the image in $\bT^*$ of
 $\diag(\lambda^{1+(\ve q)^n},\lambda^{(1+(\ve q)^n)(\ve q)^{j'}},
 \lambda^{(1+(\ve q)^n)(\ve q)^{2j'}},\ldots,\lambda^{(1+(\ve q)^n)(\ve
 q)^{(n-1)j'}})$. If we set $\mu=\lambda^{1+(\ve q)^n}$, the element
 $\diag(1,\mu^{(\ve q)^{j'}-1},\mu^{(\ve
 q)^{2j'}-1},\ldots,\mu^{(\ve q)^{(n-1)j'}-1})$
has same image in $\bT^*$.
For $\mu^{\ve q-1}=\zeta^j$ we have
 $\mu^{(\ve q)^{kj'}-1}=\zeta^{j(1+\ve q+\ldots+(\ve q)^{kj'-1})}=\zeta^{jkj'}=\zeta^k$,
the second equality since $q\equiv \ve\pmod n$.
Hence $N_{F^{2n}/c^jF}(y_0(\frac j{n(\ve q-1)((\ve q)^n+1)}))=s$,
thus we can take $y=\frac{j(q^{2n}-1)}{n(\ve q-1)((\ve q)^n+1)}y_0=
\frac{j((\ve q)^n-1)}{n(\ve q-1)}y_0$ (note 
that $\frac{(\ve q)^n-1}{n(\ve q-1)}$ is an integer).
As a character of $\bT^{c^jF}$ it maps the element
$z_0=\diag(\zeta,\ldots,\zeta)\in Z\SL_n$ to
 $\zeta_\BC^{\frac{j((\ve q)^n-1)}{n(\ve q-1)}}$. We now use:
\begin{lemma}
If $n$ is odd, $k>0$ and $n^k$ divides $\ve q-1$, we have 
$\frac{(\ve q)^n-1}{n(\ve q-1)}\equiv1\pmod{n^k}$.
\end{lemma}
\begin{proof}
 Let us write $\ve q=1+an^k$. We have $\frac{(\ve q)^n-1}{\ve q-1}=1+\ve q+\cdots+(\ve
 q)^{n-1}\equiv
n+a\frac{n(n-1)}2 n^k\pmod {n^{k+1}}$. Since $n$ is odd it divides $n(n-1)/2$,
so that
$\frac{(\ve q)^n-1}{\ve q-1}\equiv n\pmod{n^{k+1}}$ which, dividing by $n$,
gives the result.
\end{proof}
This lemma applied with $k=1$ shows that  $\hat s_j$ maps
$z_0^i$ to $\zeta_\BC^{ij}$ when $n$ is odd. If $n=2$ (thus $\ve=1$) we get
$\hat s_1(\diag(-1,-1))=(-1)^{(q+1)/2}$.
\end{proof}

We now compute the Mellin transforms $\theta_{j,k}$.
If $n$ is odd, we have by Lemma \ref{value of hat s} 
$$\theta_{j,i}(c^k)=\sum_{l\in\BZ/n\BZ} \hat
s_j(z_0^l)\theta_{z_0^l,i}(c^k)=\sum_{l\in\BZ/n\BZ}\zeta_\BC^{jl}\zeta_\BC^{lk}
=\begin{cases}
0 \text{ if }j+k\neq 0,\cr
n \text{ if } j=-k.\cr
\end{cases}
$$
We see that $\theta_{j,i}$ is a function supported by the coset
$\bT^{c^iF}.c^{-j}\subset\bT^{c^iF}\rtimes \genby c$ and is constant equal
to $n$ on this coset. By proposition \ref{Sh},
$\Sh$ maps a constant function on this coset 
to the constant function on $\bT^{c^{i+j}F}.c^{-j}$ with same value, hence
$\Sh\theta_{j,i}=\theta_{j,i+j}$.

If $n=2$ we have  $$\theta_{j,i}(c^k)=(-1)^k\hat s_j(z_0)+\hat
s_j(1)=1+\begin{cases}
(-1)^{k+(q-1)/2}&\text{ if }j=0,\\
(-1)^{k+(q+1)/2}&\text{ if }j=1.
\end{cases}$$
Thus $\theta_{j,i}$ is supported by $\bT^{c^iF}.c^{j}$ if $q\equiv 1 \pmod
4$ and by $\bT^{c^iF}.c^{1-j}$ if $q\equiv -1\pmod 4$ and is constant equal
to $2$ on its support. We get the same result as in the odd case 
for the action of $\Sh$ when $q\equiv 1\pmod 4$. If $q\equiv -1\pmod 4$, since $\Sh$ maps
functions on $\bT^{c^i}c^{1-j}$ to functions on $\bT^{c^{i+j-1}}c^{1-j}$ we get
$\Sh\theta_{j,i}=\theta_{j,i+j-1}$.
\end{proof}
Proposition \ref{Sh theta} shows that $\Sh$ preserves the space
$\CC(C_{\PGL_n}(s)^F,F,1)$.
Note that the computation made in the proof of Lemma \ref{value of hat s}
and Definition \ref{RTG non connected}
show that $R_{\bT_{c^i}}^{\bT_{c^i}\rtimes\genby c}\Id=\theta_{0,i}$ unless
$n=2$ and $q\equiv -1\pmod 4$, in which case
$R_{\bT_{c^i}}^{\bT_{c^i}\rtimes\genby c}\Id=\theta_{1,i}$.

\vskip 6mm
\subsection*{\hfill$\Sh$ on $\SL_n^F$\hfill}
\hfill\break

For computing the other side of conjecture \ref{Sh commute Jordan}, that
is $\Sh$ on $\SL_n^F$,
we first parametrise the characters in $\cup_j\CE(\SL_n^F,(s_j)_{\PGL_n^F})$,
where $s_j\in (\bT^*_{c^j})^F$ is an $F$-stable representative of the 
rational class that we parametrised above by $s\in\bT^{*c^jF}$.
We use the following notation: for
$z\in Z\SL_n$ we denote by $\Gamma_z$ the Gelfand-Graev character indexed by
$z$ as in \cite[Definition 12.3.3]{book2}; that is
$\Gamma_z=\Ind_\UF^{\SL_n^F}\psi_z$ where
$\bU$ is the unipotent radical of the Borel subgroup consisting of upper
triangular matrices and $\psi_z$
is a regular character of $\UF$ indexed by $z$. This labelling  depends
on the choice of a regular character $\psi_1$: we have $\psi_z=\lexp t\psi_1$
where $t\in\bT$ satisfies $t\inv\lexp Ft=z$.

Let $\tilde s_j\in(\tilde\bT_{c^j}^*)^F$ be a lifting of $s_j$. By 
\cite[Proposition 11.3.10]{book2}
$R_{\tilde\bT_{c^j}\cap\SL_n}^{\SL_n}(s_j)$ is
the restriction of $R_{\tilde\bT_{c^j}}^{\GL_n}(\tilde s_j)$, hence
the series $\CE(\SL_n^F,(s_j)_{\PGL_n^F})$ is the set of irreducible
components of the restrictions of 
the elements of $\CE(\GL_n^F,(\tilde s_j))$. 
Moreover since $C_{\GL_n}(\tilde s_j)$ is a torus,
the character $R_{\tilde\bT_{c^j}}^{\GL_n}(\tilde s_j)$ is
irreducible, hence is  the only character in
$\CE(\GL_n^F,(\tilde s_j))$; thus this character must be a component of the (unique)
Gelfand-Graev character of $\GL_n^F$ by \cite[Theorem 12.4.12]{book2}.
The restriction of this character to $\SL_n^F$ is equal to 
$\sum_{z\in Z\SL_n}\chi_{z,j}$ where 
$\chi_{z,j}$ is the unique irreducible component of 
the Gelfand-Graev character $\Gamma_z$ in
the series $\CE(\SL_n^F,(s_j))$ (see \cite[Corollary 12.4.10]{book2}).
Thus we have $\CE(\SL_n^F,(s_j))=\{\chi_{z,j}\mid z\in Z\SL_n\}$.
It is again more convenient to compute $\Sh$ on the basis formed by the Mellin
transforms $\chi_{j,k}:=\sum_{z\in Z\SL_n}\hat s_j(z)\chi_{z,k}$.
\begin{proposition}\label{Sh chi}
We have $\Sh\chi_{j,i}=\chi_{j,i+j}$ unless
$n=2$ (thus $\ve=1$) and $q\equiv -1\pmod 4$. In this last case
we have $\Sh\chi_{j,i}=\chi_{j,i+j-1}$.
\end{proposition}
\begin{proof}
 We note first that $\chi_{z,j}(g)$ is independent of $z$ if $C_{\SL_n}(g)$ is
connected. Indeed, $\chi_{z,j}$ and $\chi_{z',j}$ are conjugate by an element
$x\in\GL_n^F$ since they are two components of the restriction of 
an irreducible character of $\GL_n^F$. 
We have thus $\chi_{z',j}(g)=\chi_{z,j}(\lexp x g)$. We can multiply $x$ by a
central element to obtain an element $y\in \SL_n$ and, since  
$\lexp yg=\lexp xg\in\SL_n^F$, we have
 $y\inv.\lexp Fy\in C_{\SL_n}(g)$. If this centraliser is connected, using the
 Lang-Steinberg
 theorem we can multiply $y$ by an element of $C_{\SL_n}(g)$ to get a rational
element $y'$, whence $\chi_{z',j}(g)=\chi_{z,j}(\lexp {y'} g)=\chi_{z,j}(g)$.
For such an element $g$ we thus have $\chi_{j,i}(g)=0$
if $j\neq 0$.

 Since $C_{\SL_n}(g)$ is connected we have $\Sh(g)=g$ thus
$\Sh\chi_{j,i}(g)=\chi_{j,i}(g)$, in particular  $\Sh\chi_{0,i}=\chi_{0,i}$
and  if $j\neq0$ we have $(\Sh\chi_{j,i})(g)=\chi_{j,i+j}(g)=0$, whence
 $(\Sh\chi_{j,i})(g)=\chi_{j,i+j}(g)$ for all $j$.

It remains to consider the conjugacy classes of $\SL_n$ which have a
non-connected centraliser.
\begin{lemma}
When $n$ is prime the only elements of $\SL_n$
which have a non-connected centraliser are the $zu$ with $z\in Z\SL_n$ and
$u$ regular unipotent.
\end{lemma}
\begin{proof}
Let $su$ be the Jordan decomposition of an element of $\SL_n$ with $s$
 semi-simple and $u$ unipotent. We have $C_{\SL_n}(su)=C_{C_{\SL_n}(s)}(u)$. The group
 $C_{\SL_n}(s)$ is a Levi subgroup of $\SL_n$, that is the subgroup of elements of
determinant 1 in a product $\prod_{i=1}^k\GL_{n_i}$.
 Thus $C_{\SL_n}(su)$ is the subgroup
of elements with determinant 1 in a product $\prod_{i=1}^k\bH_i$ where $\bH_i$
is the centraliser of $u$ in $\GL_{n_i}$, which is connected. We claim that
 if $k>1$ the group $C_{\SL_n}(su)$
is connected. Indeed, since $n$ is prime, if $k>1$ the $n_i$ are coprime so
that there exist integers $a_i$ satisfying $\sum_{i=1}^ka_in_i=-1$. Then the
map $(h_1,\ldots,h_k)\mapsto
(h_1\lambda^{a_1},\ldots,h_k\lambda^{a_k},\lambda)$ where
$\lambda=\det(h_1h_2\ldots h_k)$ is an isomorphism from
 $\bH_1\times\cdots\times\bH_k$ to $C_{\SL_n}(su)\times\Fqbar$. Hence this last group is
 connected, thus its projection $C_{\SL_n}(su)$ is also connected.

It remains to look at the centralisers of elements $zu$ with $u$ unipotent and
 $z\in Z{\SL_n}$, that is the centralisers of unipotent elements.
 By \cite[IV, Proposition 4.1]{DMthese}, since $n$ is prime, $C_{\SL_n}(u)$ is connected
unless $u$ is regular and 
 when $u$ is regular we have $C_{\SL_n}(u)=R_u(C_{\SL_n}(u)).Z{\SL_n}$ which is not
connected.
\end{proof}
Thus to prove the proposition we have only to consider the classes of the
elements $zu$ with
 $u$ regular unipotent and $z\in Z\SL_n$. Fix a rational
regular unipotent element $u_1$. The
conjugacy classes of rational regular unipotent elements are parametrised by
 $H^1(F,C_{\SL_n}(u_1)/C_{\SL_n}\conn(u_1))=H^1(F,Z{\SL_n})=Z\SL_n$
(the last equality since $q\equiv \ve\pmod n$): a representative $u_z$ of the class
parametrised
 by the $F$-class of $z\in Z{\SL_n}$ is $\lexp tu_1$ where $t\in\bT$ satisfies
$t\inv\lexp Ft=z$.  By \cite[IV,Proposition 1.2]{DMthese} $\Sh$ maps the
 $\SL_n^F$-class
of $zu_{z'}$ to that of $zu_{zz'}$.  Now $\chi_{z,k}$ being a component of the
restriction of  the irreducible character $R_{\tilde\bT_{c^k}}^{\GL_n}(\tilde s_k)$
has central character equal to $\hat s_k$, independently of $z$, whence
$\chi_{z,k}(z'u_{z''})=\hat s_k(z')\chi_{z,k}(u_{z''})$. By
\cite[Corollary 12.3.13]{book2},
there is a family of Gauss sums $\sigma_z$ indexed by $Z{\SL_n}$ such that
$$\chi_{z,k}(u_{z''})=\sum_{z'\in H^1(F,Z{\SL_n})}
\sigma_{z''z^{\prime-1}}\scal{(-1)^{F\text{-semi-simple rank}
(\SL_n)}D(\chi_{z,k})}{\Gamma_{z'}}_{\SL_n^F},$$ where $D$ is the Curtis-Alvis
duality. If $k\neq 0$ and $F=F_+$ or if $F=F_-$ we claim that
${(-1)^{F\text{-semi-simple  rank} (\SL_n)}D(\chi_{z,k})}=\chi_{z,k}$: indeed,
in  both  cases  the  character  $R_{\tilde\bT_{c^k}}^{\GL_n}(\tilde  s_k)$ is
cuspidal since the torus $\tilde\bT_{c^k}$ is
not contained in a proper rational Levi subgroup. Indeed when $F=F_+$ and $k\neq 0$
(resp.\ $F=F_-$), the type $c^k$ (resp.\ the type $w_0c^k$) of $\bT_{c^k}$ 
with respect to a quasi-split torus is
not contained in a standard parabolic subgroup of $\fS_n$.
Thus $\chi_{z,k}$ is cuspidal as a component of the restriction
to $\SL_n^F$ of $R_{\tilde\bT_{c^k}}^{\GL_n}(\tilde s_k)$, whence our claim since the
duality on cuspidal characters is the multiplication by
$(-1)^{F\text{-semi-simple rank} (\SL_n)}$ (take $\bL={\SL_n}$ in
\cite[7.2.9]{book2}).
If $k=0$ and $F=F_+$, the characters
$\chi_{z,0}$ are the irreducible components of $R_\bT^{\SL_n}(\hat s_0)=
\Ind_{\bB^F}^{\SL_n^F}{\hat s_0}$, where $\bB$ is the Borel subgroup of upper
triangular matrices and $\hat s_0$ has been lifted to $\bB^F$.
The endomorphism algebra of this induced character is isomorphic to the
group algebra of $\BZ/n\BZ$, hence the components of
$\Ind_{\bB^F}^{\SL_n^F}{\hat s_0}$
are parametrised by the
characters of $\BZ/n\BZ$ and by \cite[Theorem B]{MG} the effect of the duality is to
multiply the parameters by the sign character of the endomorphism algebra 
which is trivial if $n$ is odd and is $-1$ if $n=2$.

If $n$ is odd or if $n=2$ and $k\neq 0$,
we thus have $\chi_{z,k}(u_{z''})= \sigma_{z''z\inv}$. 
If $n=2$ and $k=0$ we have $\chi_{z,0}(u_{z''})=-\sigma_{z'z''}$ where
$\{z,z'\}=\{1,\diag(-1,-1)\}$.

We consider first the case $n$ odd.
By Lemma \ref{value of hat s} we have
$\hat s_k(z_0^i)=\zeta_\BC^{ki}$.
Thus the values of the Mellin transforms are
$\chi_{j,k}(z'u_{z''})=\sum_i\hat s_j(z_0^i)\chi_{z_0^i,k}(z'u_{z''})=
\sum_i\hat s_j(z_0^i)\hat s_k(z')\sigma_{z''z_0^{-i}}=
\hat s_k(z')\sum_i\zeta_\BC^{ij}\sigma_{z''z_0^{-i}}$.
Let us put $z'=z_0^l$; we get  $\chi_{j,k}(z_0^lu_{z''})=
\zeta_\BC^{lk}\sum_i\zeta_\BC^{ij}\sigma_{z''z_0^{-i}}$.
Now as recalled above we have $\Sh(z'u_{z''})=z'u_{z'z''}$,
whence 
$(\Sh\chi_{j,k})(z_0^lu_{z''})=\chi_{j,k}(z_0^lu_{z_0^lz''})=
\zeta_\BC^{lk}\sum_i\zeta_\BC^{ij}\sigma_{z''z_0^{l-i}}$.
Taking $i-l$ as new variable in the summation we get
$(\Sh\chi_{j,k})(z_0^lu_{z''})=
\zeta_\BC^{lk}\sum_i\zeta_\BC^{(i+l)j}\sigma_{z''z_0^{-i}}=
\zeta_\BC^{l(j+k)}\sum_i\zeta_\BC^{ij}\sigma_{z''z_0^{-i}}=
\chi_{j,j+k}(z_0^lu_{z''})$, which gives the proposition for $n$ odd.

If $n=2$ we have $z_0=\diag(-1,-1)$.
We have  $\chi_{j,k}(zu_{z''})=
\hat s_k(z)(\hat s_j(1)\chi_{z_0^0,k}(u_{z''})+\hat
s_j(z_0)\chi_{z_0,k}(u_{z''}))$. We get
$\chi_{j,0}(zu_{z''})= \hat s_0(z)(-\sigma_{z_0z''}-\hat
s_j(z_0)\sigma_{z''})$ and 
$\chi_{j,1}(zu_{z''})= \hat s_1(z)(\sigma_{z''}+\hat
s_j(z_0)\sigma_{z_0z''})$. 
If $\hat s_j$ is the identity character then $\chi_{j,0}$ and
$\chi_{j,1}$ are invariant by $\Sh$  since
$\Sh(u_{z''})=u_{z''}$ and $\Sh(z_0u_{z''})=z_0u_{z_0z''}$.
If $\hat s_j$ is not trivial then,
$\chi_{j,0}(zu_{z''})$ and $\chi_{j,1}(zu_{z''})$ are equal if $z=1$
and  opposite if $z\neq 1$, thus, using again 
$\Sh(u_{z''})=u_{z''}$ and $\Sh(z_0u_{z''})=z_0u_{z_0z''}$ we see that
$\chi_{j,0}$  and $\chi_{j,1}$ are exchanged by $\Sh$.
By Lemma \ref{value of hat s}, if 
$q\equiv 1\pmod 4$ the character $\hat s_0$ is trivial and we get the same
result as in the odd case. If $q\equiv -1\pmod 4$, the character $\hat s_1$
is trivial and $\Sh$ exchanges $\chi_{0,0}$ and $\chi_{0,1}$ and fixes
$\chi_{1,0}$ and $\chi_{1,1}$, which is the announced result.
\end{proof}
We can now state:
\begin{proposition}
For $s=\diag(1,\zeta,\ldots,\zeta^{n-1})$,
the bijection $J:\chi_{j,i}\mapsto\theta_{j,i}$ from $\BC\CE(\SL_n^F,(s))$ to 
 $\BC\CE(C_{\PGL_n}(s)^F,F,1)$ restricts on characters to
a refinement of the Jordan decomposition
which satisfies Conjecture \ref{Sh commute Jordan}.
\end{proposition}
\begin{proof}Propositions \ref{Sh theta} and \ref{Sh chi} give
the commutation of $J$ with $\Sh$. It
 remains to show that $J(R_{\bT_{c^j}}^{\SL_n}(s_j))=
R_{\bT^*_{c^j}}^{\bT^*_{c^j}\rtimes\genby c}(1)$.
As we have noticed after the proof of Proposition \ref{Sh theta} 
we have $R_{\bT_{c^j}}^{\bT_{c^j}\rtimes\genby c}\Id=\begin{cases}\theta_{0,j}&\text{ for }n
\text{ odd or } q\equiv 1\pmod 4\\
\theta_{1,j}&\text{ for }n=2,\,q\equiv-1\pmod 4\end{cases}$.
On the other hand for $n$ odd or $q\equiv 1\pmod 4$,
we have $\chi_{0,j}=\sum_z\hat
s_0(z)\chi_{z,j}=\sum_z\chi_{z,j}$ since in that case $\hat s_0$ is the trivial character.
For $n=2,\,q\equiv -1\pmod 4$,
we have $\chi_{1,j}=\sum_z\hat
s_1(z)\chi_{z,j}=\sum_z\chi_{z,j}$ since in that case $\hat s_1$ is the trivial character.
By definition of $\chi_{z,j}$ we have
$\sum_z\chi_{z,j}=R_{\bT_{c^j}}^{\SL_n}(s_j)$, whence the proposition.
\end{proof}

We thank C\'edric Bonnaf\'e for having suggested improvements to a previous
version of this text.

\end{document}